\documentclass[11pt,reqno]{amsart}

\usepackage{amsmath,amssymb,amsthm,mathtools,mathrsfs}
\usepackage{epsfig}
\usepackage{mathdots}
\usepackage{hyperref}
\hypersetup{hidelinks}
\usepackage{fullpage}
\usepackage{color}
\usepackage{enumitem}
\usepackage{array,longtable}

\newcolumntype{C}{>{$}c<{$}}

\theoremstyle{plain}
\newtheorem{theorem}{Theorem}[section]
\newtheorem{corollary}{Corollary}[section]

\newtheorem{lemma}{Lemma}[section]
\newtheorem{proposition}{Proposition}[section]

\theoremstyle{definition}
\newtheorem{definition}{Definition}[section]
\newtheorem{remark}{Remark}[section]

\numberwithin{equation}{section}

\def\Z{\mathbb{Z}}
\def\Q{\mathbb{Q}}
\def\R{\mathbb{R}}
\def\C{\mathbb{C}}
\def\F{\mathbb{F}}

\newcommand{\OO}{\mathcal{O}}
\newcommand{\End}{\operatorname{End}}
\newcommand{\Aut}{\operatorname{Aut}}
\newcommand{\Spec}{\operatorname{Spec}}
\newcommand{\Trd}{\operatorname{Trd}}
\newcommand{\Nrd}{\operatorname{Nrd}}
\newcommand{\tr}{\operatorname{tr}}
\newcommand{\ord}{\operatorname{ord}}
\newcommand{\Norm}{\operatorname{Norm}}
\newcommand{\Res}{\operatorname{Res}}
\newcommand{\length}{\operatorname{length}}
\newcommand{\Diff}{\operatorname{Diff}}
\newcommand{\eps}{\varepsilon}
\newcommand{\Aideal}{\mathfrak A}
\newcommand{\qideal}{\mathfrak q}
\newcommand{\lideal}{\mathfrak l}

\newcommand{\mdeg}{\operatorname{deg}}

\begin{document}

\title{Difference of the modular function $\omega_{2}(\tau)$, revisited}

\author{Wei-Lun Tsai and Dongxi Ye}

\address{Department of Mathematics, University of South Carolina,
1523 Greene St, LeConte College, Room 450,
Columbia, SC 29208}

\email{weilun@mailbox.sc.edu}

\address{Beijing Normal--Hong Kong Baptist University,
Zhuhai 519082, Guangdong, People's Republic of China}

\email{dongxiye@bnbu.edu.cn}



\thanks{Dongxi Ye was supported by the Guangdong Basic and Applied Basic
Research Foundation (Grant No.~2024A1515030222) and the BNBU Start-up
Research Fund (Grant No.~R0700157-26).}

\begin{abstract}
Adapting the analytic method of Gross and Zagier, Roskam 
proved a prime-factorization formula
for the norm of the difference of two level-two Weber singular moduli.
Independently, Yang and Yin obtained an equivalent formula using
Borcherds lifts. 
More precisely, the formula
concerns the norm of
\[
 \omega_{2}\left(\frac{-1+\sqrt{d_{1}}}{2}\right)
 -
 \omega_{2}\left(\frac{-1+\sqrt{d_{2}}}{2}\right)
\]
for coprime negative fundamental quadratic discriminants
$d_{1},d_{2}\equiv1\pmod 8$, where
\[
 \omega_{2}(\tau)
 =
 2^{12}\frac{\eta(2\tau)^{24}}{\eta(\tau)^{24}},
\]
and $\eta(\tau)$ denotes the Dedekind eta function. In this work, we revisit this formula from the perspective of arithmetic intersection
theory and give a new proof. 
\end{abstract}

\maketitle

\allowdisplaybreaks

\section{Introduction}

Let $d_1,d_2<0$ be distinct, coprime fundamental discriminants
satisfying
\begin{equation}
\label{eq:hypotheses}
 d_1\equiv d_2\equiv1\pmod 8,
 \qquad
 D=d_1d_2>0.
\end{equation}
Set
\[
 K_i=\Q(\sqrt{d_i}),
 \qquad
 \OO_i=\Z\left[\frac{-1+\sqrt{d_i}}{2}\right],
\]
and define
\begin{equation*}
 \omega_2(\tau)
 =
 2^{12}\frac{\eta(2\tau)^{24}}{\eta(\tau)^{24}},
 \qquad
 \beta_i
 =
 \omega_2\left(\frac{-1+\sqrt{d_i}}{2}\right).
\end{equation*}
Here $\eta(\tau)$ denotes the Dedekind eta function.

Using different analytic methods, Roskam \cite[Theorem~5]{Roskam} and Yang and
Yin \cite[Theorem~1.3]{YY} independently established a prime-factorization formula
for the rational norm
$
 \left|
 \Norm_{\Q(\beta_1,\beta_2)/\Q}(\beta_1-\beta_2)
 \right|.
$
Their work was inspired by the formula of Gross and Zagier \cite{GZ}
for the norm of the difference of two singular values of the modular
$j$-invariant.

To state the formula, we first introduce the arithmetic functions that occur in the
formula. For a rational prime $\ell$ satisfying
\(
 \left(\frac{D}{\ell}\right)\ne-1,
\)
choose an index $i\in\{1,2\}$ such that $\ell\nmid d_i$ and define
\begin{equation}
\label{eq:epsilon-prime}
 \eps(\ell)
 =
 \left(\frac{d_i}{\ell}\right).
\end{equation}
Here and below, $\left(\frac{\cdot}{\cdot}\right)$ denotes the
Kronecker--Jacobi symbol. Such an index $i$ exists, since
$(d_1,d_2)=1$, and the definition in \eqref{eq:epsilon-prime} is
independent of the choice of $i$.

Extend $\eps$ completely multiplicatively to positive integers all of
whose prime divisors $\ell$ satisfy
\(
 \left(\frac{D}{\ell}\right)\ne-1.
\)
We call such integers admissible. For an admissible positive integer
$n$, define
\begin{equation}
\label{eq:rhoepsilon}
 \rho(n)=\sum_{a\mid n}\eps(a).
\end{equation}
We also set $\rho(n)=0$ whenever $n$ is not a positive admissible
integer.

Furthermore, define
\begin{equation}
 X_D
 =
 \left\{
 x\in\Z:
 0<x<\sqrt D,\quad x^2\equiv D\pmod{16}
 \right\},
 \quad\mbox{and}\quad
 m_x=\frac{D-x^2}{16}
 \,\,\,\mbox{for\, $x\in X_D$}.\nonumber
\end{equation}
It is immediate from the definition that $m_x$ is admissible. For an
admissible positive integer $m$, put
\begin{equation}
 F(m)
 =
 \prod_{\substack{a,b>0\\ab=m}}a^{\eps(b)}.\nonumber
\end{equation}
For any admissible $m$ and any prime $p$, we have
\begin{equation}
 \ord_p F(m)
 =
 \sum_{r\ge1}\rho(m/p^r).\nonumber
\end{equation}
 Indeed, one has
\[
\begin{aligned}
 \ord_p F(m)
 &=\sum_{ab=m}\eps(b)\ord_p(a) =\sum_{r\ge1}
   \sum_{\substack{ab=m\\p^r\mid a}}\eps(b) =\sum_{r\ge1}\sum_{b\mid m/p^r}\eps(b) =\sum_{r\ge1}\rho(m/p^r).
\end{aligned}
\]

We can now state the prime-factorization formula.

\begin{theorem}[Roskam, Yang--Yin]
\label{thm:norm-formula}
Under \eqref{eq:hypotheses}, one has 
\begin{equation}
\label{eq:main-norm-formula}
 \left|
 \Norm_{\Q(\beta_1,\beta_2)/\Q}(\beta_1-\beta_2)
 \right|
 =
 \prod_{x\in X_D}F(m_x).
\end{equation}
Equivalently, for any prime $p$,
\begin{equation}
\label{porderformula}
 \ord_p
 \left|
 \Norm_{\Q(\beta_1,\beta_2)/\Q}(\beta_1-\beta_2)
 \right|
 =
 \sum_{x\in X_D}\sum_{r\ge1}\rho(m_x/p^r).    
\end{equation}

\end{theorem}

As noted above, formula \eqref{eq:main-norm-formula} was proved
independently by Roskam and by Yang and Yin using different analytic
methods. Roskam's argument adapts the analytic method of Gross and
Zagier for the prime factorization of singular values of the modular
$j$-invariant \cite{GZ}.

Gross and Zagier also gave a conceptual algebraic proof of their
formula. This raises the natural question of whether an analogous
algebraic proof exists for \eqref{eq:main-norm-formula}. The level-two
case is more subtle. Their algebraic argument uses a correspondence
involving CM points, optimal embeddings into maximal quaternion orders,
and local intersections governed by a single special endomorphism. In
the present setting, the level structure introduces a nonmaximal order,
a subgroup that is not stable under the full CM order, and additional
deformation data. Thus, the argument does not apply directly, although
its analytic counterpart can be adapted, as in Roskam's work.

Our method overcomes this difficulty by computing the relevant
arithmetic intersections directly. The proof has two normalization
issues that must be kept separate. First, the moduli problem is
naturally a stack, so objects are counted with automorphism weights.
Second, after one CM type is fixed, the second CM action has two
possible orientations, corresponding to the signed invariants $x$ and
$-x$. These two contributions cancel the automorphism factor. The
remaining multiplicity is the deformation length, which we express as
a sum over finite lifting levels.

\section{The genus character and the divisor sum}
\label{sec:genus-sign}

In this section, we establish basic properties of the arithmetic
functions $\eps(\cdot)$ and $\rho(\cdot)$ as defined in  \eqref{eq:epsilon-prime} and \eqref{eq:rhoepsilon}. We write 
\(
 F=\Q(\sqrt D)
\)
and
\(
 E=\Q(\sqrt{d_1},\sqrt{d_2}).
\)
Then $E/F$ is a quadratic CM extension. Let
\(
 \chi=\chi_{E/F}
\)
be the associated quadratic character. The following lemma is classical.

\begin{lemma}
\label{lem:genusfield}
The extension $E/F$ is unramified at every finite prime and
ramified at both real places. If $a\in F^\times$, then
\begin{equation}
\label{eq:principal-character}
 \chi(a\OO_{F})
 =
 \operatorname{sgn}(\sigma_1(a))
 \operatorname{sgn}(\sigma_2(a)),
\end{equation}
where
\(
 \sigma_1,\sigma_2:F\hookrightarrow\R
\)
are the two real embeddings.
\end{lemma}

\begin{proof}
See, e.g., \cite[VI.2]{Gras}.
\end{proof}

Fix any $x\in X_D$ and write
\(
 t_x=\frac{x+\sqrt D}{2}.
\)
Since $D\equiv1\pmod8$, the prime $2$ splits in $F$. Since the trace
of $t_x$ is odd, exactly one prime above $2$ divides $t_x$. Denote this
prime by $\qideal_x$ and define
\begin{equation}
 \Aideal_x=t_x\OO_{F}\qideal_x^{-2}.\nonumber
\end{equation}

\begin{lemma}
\label{lem:prime-dict}
The ideal $\Aideal_x$ is integral and
\begin{equation}
\label{eq:norm-Aideal}
 \Norm_{F/\Q}(\Aideal_x)=m_x.
\end{equation}
For any rational prime $\ell\mid m_x$, there is a unique prime
$\lideal_{x,\ell}$ of $F$ above $\ell$ that divides $\Aideal_x$, and
\begin{equation}
\label{eq:prime-dictionary}
 \Norm(\lideal_{x,\ell})=\ell,
 \qquad
 \ord_{\lideal_{x,\ell}}(\Aideal_x)
 =
 \ord_\ell(m_x),
 \qquad
 \chi(\lideal_{x,\ell})=\eps(\ell).
\end{equation}
\end{lemma}

\begin{proof}
By definition,
\[
 \Norm_{F/\Q}(t_x)
 =
 \frac{x^2-D}{4}
 =
 -4m_x.
\]
The unique prime $\qideal_x$ above $2$ that divides $t_x$ occurs with
valuation
\(
 2+\ord_2(m_x).
\)
Dividing by $\qideal_x^2$ proves the integrality of $\Aideal_x$ and
gives \eqref{eq:norm-Aideal}.

Suppose that an odd prime $\ell$ divides $m_x$. Then
\(
 D\equiv x^2\pmod\ell,
\)
so $\ell$ is split or ramified in $F$, but never inert. Exactly one
prime above $\ell$ divides $t_x$. Indeed, if two distinct primes above
$\ell$ divided $t_x$, then $\ell$ would divide both $x$ and $\sqrt D$.
If $\ell\nmid D$, this is impossible, and if $\ell\mid D$, there is only
one ramified prime. The norm identity gives the first two assertions in
\eqref{eq:prime-dictionary}.

If $\ell\nmid 2D$, then
\(
 \left(\frac{d_1}{\ell}\right)
 =
 \left(\frac{d_2}{\ell}\right)
 =
 \eps(\ell),
\)
and this common symbol determines whether the relevant prime of $F$
splits in $E$. If $\ell\mid d_1$, then
\(
 E=F(\sqrt{d_2}),
\)
and $\left(\frac{d_2}{\ell}\right)=\eps(\ell)$ gives the same
criterion. The case $\ell\mid d_2$ is symmetric.

Finally, each $d_i$ is a square in $\Q_2$. Thus, every prime lying above $2$
splits in $E/F$, and
\(
 \chi(\lideal_{x,2})=1=\eps(2).
\)
\end{proof}

By the preceding lemmas, one can find that the value of $\varepsilon$ at $m_{x}$ is always~$-1$ for $x\in X_{D}$.

\begin{proposition}
\label{prop:genus-sign}
For any $x\in X_D$, we have
\begin{equation}
 \eps(m_x)=-1.\nonumber
\end{equation}
\end{proposition}

\begin{proof}
The two real conjugates of $t_x$ have opposite signs since
$0<x<\sqrt D$. Lemma~\ref{lem:genusfield} therefore gives
\(
 \chi(t_x\OO_{F})=-1.
\)
The factor $\qideal_x^{-2}$ does not change the value of the character,
so
\(
 \chi(\Aideal_x)=~-1.
\)
Factoring $\Aideal_x$ and applying Lemma~\ref{lem:prime-dict}, we obtain that 
\(
 -1
 =
 \chi(\Aideal_x)
 =
 \prod_{\ell\mid m_x}
 \eps(\ell)^{\ord_\ell(m_x)}
 =
 \eps(m_x).
\)
\end{proof}

Finally, the following lemma regarding values of $\rho$ at prime-power inputs will also be useful.

\begin{lemma}
\label{lem:prime-power}
For an admissible integer $n$ and a prime $\ell$, we have
\begin{equation}
\label{eq:rho-prime-powers}
 \rho(\ell^a)
 =
 \begin{cases}
 a+1,&\eps(\ell)=1,\\
 1,&\eps(\ell)=-1\text{ and }a\text{ is even},\\
 0,&\eps(\ell)=-1\text{ and }a\text{ is odd}.
 \end{cases}
\end{equation}
In consequence, if
\(
 m=p^e n,\, p\nmid n,\, \eps(p)=-1,\, \eps(n)=1,
\)
then, for $1\le r\le e$,
\begin{equation}
\label{eq:rho-remove-prime}
 \rho(m/p^r)
 =
 \begin{cases}
  \rho(n),&e-r\text{ is even},\\
  0,&e-r\text{ is odd}.
 \end{cases}
\end{equation}
\end{lemma}

\begin{proof}
Equation \eqref{eq:rho-prime-powers} follows from the finite geometric
sum
\[
 \rho(\ell^a)
 =
 \sum_{b=0}^{a}\eps(\ell)^b.
\]
Equation \eqref{eq:rho-remove-prime} follows from multiplicativity.
\end{proof}

\section{The $2$-adic unit}
\label{sec:two-adic}

It is known, see for example \cite{Roskam}, that $\omega_2(\tau)$ is a
uniformizer for the genus-zero modular curve $X_0(2)$. It satisfies
the relation
\begin{equation}
\label{eq:modular-equation}
 j(\tau)
 =
 \frac{(\omega_2(\tau)+16)^3}{\omega_2(\tau)}
\end{equation}
with the modular $j$-invariant. Equivalently, $\omega_2(\tau)$ is a
root of
\begin{equation}
\label{eq:omega-cubic}
 X^3+48X^2+(768-j)X+4096=0.
\end{equation}
Moreover, $\beta_i$ is an algebraic integer whose $\Q$-conjugates are
\[
 \omega_2\left(\frac{-b+\sqrt{d_i}}{2a}\right),
 \qquad
 2\nmid a>0,
 \qquad
 b^2-4ac=d_i.
\]

The following elementary moduli interpretation of the condition that
$a$ is odd will be useful.

\begin{lemma}
\label{lem:odd-leading-line}
Let $[a,b,c]$ be a primitive positive definite binary quadratic form of
fundamental discriminant $d<0$. Write
\[
 \tau=\frac{-b+\sqrt d}{2a},
 \qquad
 E_\tau=\C/(\Z\tau+\Z),
 \qquad
 C_\tau=\left\langle\frac12\right\rangle\subset E_\tau[2].
\]
Then
\(
 \End(E_\tau)=\Z[a\tau]=\OO_d.
\)
The subgroup $C_\tau$ is stable under $\OO_d$ if and only if $a$ is
even. If $a$ is odd, its stabilizer in $\OO_d$ is
\(
 \Z+2\OO_d.
\)
In consequence, the $\Q$-Galois orbit of
\(
 \omega_2\left(\frac{-1+\sqrt d}{2}\right)
\)
is precisely the locus on which the cyclic subgroup of order $2$ is
not stable under the maximal CM order. This property is preserved under
specialization in every odd characteristic.
\end{lemma}

\begin{proof}
Since $d$ is fundamental and odd,
\(
 a\tau=\frac{-b+\sqrt d}{2}
\)
generates $\OO_d$ over $\Z$. Write an element of $\OO_d$ as
\(
 m+na\tau.
\)
Its action on the generator $1/2$ of $C_\tau$ is
\(
 \frac{m}{2}+\frac{na}{2}\tau
 \pmod{\Z\tau+\Z}.
\)
This lies in $\{0,1/2\}$ modulo the period lattice if and only if
$na$ is even. Hence, every element of $\OO_d$ preserves $C_\tau$ exactly
when $a$ is even. If $a$ is odd, the elements that preserve $C_\tau$
are those for which $n$ is even, namely the elements of
$\Z+2\OO_d$.

In odd   characteristic, $E[2]$ is finite $\acute{e}$tale . Specialization
therefore identifies its geometric points and the actions of the
specialized CM endomorphisms. Thus, stability and nonstability of the
distinguished subgroup are unchanged under specialization.
\end{proof}

The next lemma specifies the $2$-adic unit'ness of the roots of~\eqref{eq:omega-cubic} when $j$ is a~$2$-adic unit.

\begin{lemma}
\label{lem:branches}
Let $j$ be a $2$-adic unit. The three roots of
\eqref{eq:omega-cubic} have valuations $12,0,0$. The unique root of
positive valuation is the branch corresponding to the connected
canonical subgroup of an ordinary elliptic curve.
\end{lemma}

\begin{proof}
The valuations of the coefficients in \eqref{eq:omega-cubic}, listed
from the constant term to the leading term, are
\(
 12,\, 0,\, 4,\, 0,
\)
since $768-j$ is a unit. The lower Newton polygon has one segment of
slope $-12$ and horizontal length one, followed by a horizontal segment
of length two. Hence, the root valuations are $12,\,0,\,0$.

For an ordinary elliptic curve in characteristic $2$, the finite flat
group scheme $E[2]$ has a unique connected subgroup of rank two. This
is the canonical subgroup. It is the branch that specializes to the
connected component in the ordinary local model of $X_0(2)$. Since
\(
 \omega_2=2^{12}q+O(q^2)
\)
at the cusp $\infty$, this branch has positive parameter valuation,
whereas the two $\acute{e}$tale  branches have unit parameter. This is the
standard ordinary local model described in
\cite[Chapters~IV--V]{DR} and \cite[Chapters~12--13]{KM}.
\end{proof}

\begin{proposition}
\label{prop:unit}
For any embedding of a conjugate of $\beta_i$ into $\overline{\Q}_2$,
the image is a $2$-adic unit.
\end{proposition}

\begin{proof}
Fix a conjugate and realize it as $\omega_2(E,C)$ using
\cite[Proposition~8]{Roskam}. After a finite extension of $\Q_2$, the
CM elliptic curve has good reduction. Since $2$ splits in $K_i$, the
reduction is ordinary by the CM reduction criterion. In consequence,
its $j$-invariant is a $2$-adic unit, since the only supersingular
$j$-invariant in characteristic $2$ is $0$.

The connected canonical subgroup is stable under $\OO_i$, since
\(
 \OO_i\otimes\Z_2\cong\Z_2\oplus\Z_2
\)
decomposes the $2$-divisible group into its connected and $\acute{e}$tale  CM
factors. By Lemma~\ref{lem:odd-leading-line}, the subgroup $C$ in the
orbit with odd leading coefficient is not stable under the maximal
order. Its stabilizer is the index-two order
\(
 \Z+2\OO_i.
\)
Hence, $C$ is not the connected canonical subgroup. Lemma~\ref{lem:branches}
therefore places $\omega_2(E,C)$ on one of the two unit branches.
\end{proof}

Now we are ready to prove the formula~\eqref{porderformula} for~$p=2$.

\begin{proposition}
\label{prop:orderat2}
Under \eqref{eq:hypotheses}, we have
\begin{equation}
\label{eq:two-adic-unit}
 \ord_2
 \left|
 \Norm_{\Q(\beta_1,\beta_2)/\Q}(\beta_1-\beta_2)
 \right|
 =
 0
 =
 \sum_{x\in X_D}\sum_{r\ge1}\rho(m_x/2^r).
\end{equation}
\end{proposition}

\begin{proof}
Let
\(
 L=\Q(\beta_1,\beta_2)
\)
and fix an embedding
$
 \sigma:L\hookrightarrow\overline{\Q}_2.
$
By Proposition~\ref{prop:unit}, both $\beta_1^\sigma$ and
$\beta_2^\sigma$ are $2$-adic units. Reducing \eqref{eq:modular-equation}
modulo $2$ gives, for a unit root, we have
\begin{equation}
\label{eq:j-beta-mod2}
 \overline{j}=\overline{\beta}^{\,2}.
\end{equation}

Suppose that
\(
 \overline{\beta_1^\sigma}
 =
 \overline{\beta_2^\sigma}.
\)
Then \eqref{eq:j-beta-mod2} shows that the two ordinary reductions have
the same $j$-invariant. Hence, they are isomorphic over
$\overline{\F}_2$ to a single ordinary elliptic curve $E$.
Specialization of endomorphisms under good reduction is injective.
After transporting one action across the isomorphism, both $K_1$ and
$K_2$ would embed into $\End^0(E)$. The rational endomorphism algebra
of an ordinary elliptic curve over $\overline{\F}_2$ is a single
imaginary quadratic field, so it cannot contain two distinct quadratic
subfields. This contradicts $d_1\ne d_2$.

Thus
\(
 \beta_1^\sigma-\beta_2^\sigma
\)
is a unit for any $\sigma$. Taking the product over all embeddings of
$L$ proves the first equality in \eqref{eq:two-adic-unit}. This argument
does not require linear disjointness.

On the product side, $\eps(2)=1$ and
\(
 \eps(m_x)=-1
\)
by Proposition~\ref{prop:genus-sign}. Hence, some prime $\ell$ with
$\eps(\ell)=-1$ occurs to an odd power in $m_x$, and this remains true
after division by any power of $2$. Lemma~\ref{lem:prime-power}
therefore gives
\(
 \rho(m_x/2^r)=0
\)
for any $r\ge1$.
\end{proof}

\section{CM embeddings and special endomorphism cycles}
\label{sec:KY-dictionary}

Fix an odd prime $p$ and $x\in X_D$. Suppose that there is an  $i\in\{1,2\}$ such that
\begin{equation}
\label{eq:base-choice}
 p\nmid d_i,
 \qquad
 \left(\frac{d_i}{p}\right)=-1,
\end{equation}
and let $j\ne i$. As one shall see in Lemma~\ref{lem:lifting-support} that on every nonempty component of $$
 \Spec\left({\Z_p[t,t^{-1}]}/{(P_1(t),P_2(t))}\right),
 $$
 such an index~$i$ always exists.
Write
\(
 d=d_i,
 \,
 K=K_i,
 \,
 \OO=\OO_i,
 \,
 s=\iota_i(\sqrt d),
 \)
 where $\imath_{i}:\mathcal{O}\hookrightarrow{\rm End}(E)$ denotes the maximal CM embedding of $\mathcal{O}$,
 and
 \(
 \partial=\sqrt d\,\OO.
\)
For
\(
 \xi\in\{x,-x\},
\)
define
\begin{equation}
 \lambda=-\frac{2}{\sqrt d}\pmod{\OO}
 \in\partial^{-1}/\OO,
 \qquad
 r_\xi
 =
 \frac{\sqrt d+\xi}{2\sqrt d}\pmod{\OO}
 \in\partial^{-1}/\OO.\nonumber
\end{equation}
The second inclusion follows since
\(
 \frac{\sqrt d+\xi}{2}\in\OO
\)
for odd $\xi$. At every prime dividing $d$, the integer $2$ is a unit,
so $\lambda$ generates $\partial^{-1}/\OO$. In the notation of
\cite[equation~(1.4)]{KY}, this gives
\begin{equation}
\label{eq:partial-lambda}
 \partial_\lambda=\partial,
 \qquad
 \Delta(\lambda)=N(\partial_\lambda)=|d|.
\end{equation}

The special endomorphism cycle
\(
 \mathcal Z(m_x;\OO_i,\lambda,r_\xi)
\)
classifies triples $(E,\iota_i,u)$ satisfying
\begin{equation}
\label{eq:KY-conditions}
 u\iota_i(a)=\iota_i(\overline a)u,
 \qquad
 \deg(u)=m_x,
 \qquad
 r_\xi+u\lambda\in\End(E).
\end{equation}
These triples are characterized in the following proposition.

\begin{proposition}
\label{prop:reconstruct}
Fix an odd prime $p$. In the category of deformations over
$p$-nilpotent bases of a supersingular elliptic curve with an optimal
$\OO_i$-action, the following two finite moduli problems are 
equivalent:
\begin{enumerate}[label=(\alph*)]
\item ordered pairs of maximal CM embeddings
\[
 \iota_i:\OO_i\hookrightarrow\End(E),
 \qquad
 \iota_j:\OO_j\hookrightarrow\End(E)
\]
such that
\(
 \frac12\Trd(s_is_j)=\xi
\)
and
\(
 u=\frac{s_is_j-\xi}{4}\in\End(E);
\)

\item triples
\(
 (E,\iota_i,u)\in
 \mathcal Z(m_x;\OO_i,\lambda,r_\xi).
\)
\end{enumerate}
The two constructions are inverse and preserve automorphisms.
\end{proposition}

\begin{proof}
Start with a tuple in (a). For trace-zero elements of a quaternion
algebra, we have that 
\(
 s_is_j+s_js_i=\Trd(s_is_j)=2\xi.
\)
It follows that
\(
 us_i=-s_iu,
\)
so $u$ is $\OO_i$-antilinear. By assumption, $u$ is integral.
Moreover,
\(
 \Trd(u)=0.
\)
Since
\(
 \Trd(s_is_j)=2\xi,
 \)
 and
 \(
 \Nrd(s_is_j)=D,
\)
we have
\(
 \Nrd(s_is_j-\xi)=D-\xi^2.
\)
The reduced characteristic polynomial therefore gives
\(
 u^2=-\Nrd(u),
\)
and
\(
 \Nrd(u)=\frac{D-\xi^2}{16}=m_x.
\)
Thus, $\deg(u)=m_x$.

Let
\(
 A_j=\frac{1+s_j}{2}.
\)
Using $us_i=-s_iu$ and $s_i^{-1}=s_i/d_i$, we obtain
\begin{align}
 r_\xi+u\lambda
 &=\frac12+\frac{\xi}{2s_i}-\frac{2u}{s_i}
 \label{eq:reconstruct-one}\\
 &=\frac12+\frac{\xi}{2s_i}+2s_i^{-1}u
 =\frac{1+s_j}{2}
 =A_j. \label{eq:reconstruct-two}
\end{align}
Thus, the integrality condition in \eqref{eq:KY-conditions} is exactly
the integrality of the second maximal CM generator.

Conversely, start with a triple in (b) and put
\begin{equation}
\label{eq:recover-sj}
 A_j=r_\xi+u\lambda,
 \qquad
 s_j=2A_j-1=s_i^{-1}(4u+\xi).
\end{equation}
An antilinear endomorphism has reduced trace zero and satisfies
\(
 u^2=-\deg(u)=-m_x.
\)
Since $us_i=-s_iu$, then we have that 
\[
\begin{aligned}
 s_j^2
 &=s_i^{-2}(-4u+\xi)(4u+\xi)=d_i^{-1}(\xi^2-16u^2)=d_i^{-1}(x^2+16m_x)=d_j.
\end{aligned}
\]
Moreover,
\(
 A_j^2-A_j+\frac{1-d_j}{4}=0.
\)
Thus, $A_j$ defines an embedding of the maximal order
\(
 \OO_j=\Z\left[\frac{1+\sqrt{d_j}}{2}\right]
\)
into $\End(E)$. This embedding is optimal since $\OO_j$ is maximal.
Finally, we have 
\(
 s_is_j=4u+\xi,
 \)
 and
 \(
 \frac12\Trd(s_is_j)=\xi.
\)
Equations \eqref{eq:reconstruct-one} and \eqref{eq:recover-sj} show
that the two constructions are inverse. They also preserve morphisms
and automorphisms.
\end{proof}

\subsection{Recovery of the nonstable subgroup of order $2$}

Since $p$ is odd, $E[2]$ is finite $\acute{e}$tale  and its geometric points form
a two-dimensional $\F_2$-vector space. Write 
\(
 A_i=\frac{1+s_i}{2}.
\)
The split algebra
\(
 \OO_i\otimes\F_2\cong\F_2\times\F_2
\)
acts faithfully on $E[2]$. Indeed, if $\alpha\in\OO_i$ annihilates
$E[2]$, then $\alpha$ factors through multiplication by $2$, so
\(
 \alpha=2\beta
\)
for some $\beta\in\End(E)$. Optimality implies that
$\beta\in\OO_i$. Hence, the reduction of $A_i$ modulo $2$ is a
nontrivial rank-one idempotent.

\begin{lemma}
\label{lem:lines}
For a triple specified in Proposition~\ref{prop:reconstruct}, we have 
\begin{equation}
\label{eq:actions-mod2}
 A_j\equiv A_i+cI\pmod2,
 \qquad
 c\equiv\frac{\xi-d_i}{2}\pmod2.
\end{equation}
Hence, $A_i$ and $A_j$ have the same two stable lines in $E[2]$
and exactly one common nonstable line.
\end{lemma}

\begin{proof}
At $2$, the element $s_i$ is a unit and
\(
 \lambda=-\frac{2}{s_i}\in2\OO_{i,2}.
\)
Hence
\(
 u\lambda\equiv0\pmod2.
\)
Also,
\begin{equation}
\label{eq:r-minus-A}
 r_\xi-A_i=\frac{\xi-d_i}{2s_i}.
\end{equation}
Since
\(
 s_i\equiv1\pmod2,
\)
equations \eqref{eq:reconstruct-one} and \eqref{eq:r-minus-A} give
\eqref{eq:actions-mod2}.

Adding a scalar to an endomorphism does not change its invariant
subspaces. A nontrivial idempotent on $\F_2^2$ has two eigenspaces, its kernel and image, while the third line is not stable.
\end{proof}

Upon this, another interpretation of~$ \mathcal Z(m_x;\OO_i,\lambda,r_\xi)$ can be realized.

\begin{corollary}
\label{cor:level2}
In the setting of Proposition~\ref{prop:reconstruct}, the common cyclic
subgroup of order $2$ that is not stable under either maximal CM order
is uniquely determined by the triple $(E,\iota_i,u)$. Consequently, the
finite moduli problem of level-two CM configurations with signed
invariant $\xi$ is equivalent to
\(
 \mathcal Z(m_x;\OO_i,\lambda,r_\xi).
\)
\end{corollary}

\begin{proof}
By Lemma~\ref{lem:lines}, the two maximal CM actions have the same two
stable lines in $E[2]$ and exactly one common nonstable line. Hence this
line determines the unique cyclic subgroup of order $2$ that is not
stable under either maximal CM order.

Since $2$ is invertible in the category of $p$-nilpotent deformations,
the group scheme $E[2]$ is finite $\acute{e}$tale. Therefore, the uniquely
determined subgroup lifts uniquely through every nilpotent thickening.
Thus adding the level-two subgroup introduces no additional choice or
multiplicity, and the corollary follows.
\end{proof}

\subsection{Point counts and local factors}
\label{sec:point-count}

Keep the choice of $i$ from \eqref{eq:base-choice}, and write
\(
 d=d_i,\)
 and
 \(
 m=m_x.
\)
Take the fractional ideal $\mathfrak a=\OO_i$ in
\cite[Definition~2.1]{KY}. By \eqref{eq:partial-lambda}, we have
\begin{equation}
 a_0=N(\partial_\lambda^{-1}\mathfrak a)=\frac1{|d|},
 \qquad
 \eta=a_0m=\frac{m}{|d|}.\nonumber
\end{equation}
Note that $\eta$ is the Fourier index appearing in
\cite[Theorem~3.6]{KY}.

Let $\chi_i$ denote the quadratic character of $K_i/\Q$. For the
coherent quadratic space and Schwartz function in
\cite[Equations~(3.4)--(3.5)]{KY}, define
\begin{equation}
 e_\ell(m)
 =
 \gamma\!\left(V_\ell^{(p)}\right)^{-1}
 W_{\eta,\ell}^{*}
 \left(0,\widetilde\varphi_\ell^{(p)}\right),\nonumber
\end{equation}
where $\gamma(V_\ell^{(p)})$ is the local Weil index. Then we show that $e_{\ell}(m)$ can be compactly expressed in terms of~$\varepsilon(\ell)$.

\begin{proposition}
\label{prop:local-factors}
Assume that $p$ is the coherent support prime in
\cite[Theorem~3.6]{KY}. For any finite prime $\ell\ne p$, we have
\begin{equation}
\label{eq:finite-local-factor}
 e_\ell(m)
 =
 \sum_{a=0}^{\ord_\ell(m)}\eps(\ell)^a.
\end{equation}
At the coherent prime $p$, one has 
\(
 e_p(m)=1.
\)
\end{proposition}

\begin{proof}
First suppose that $\ell\nmid d$. The local Schwartz function is the
characteristic function of $\OO_{i,\ell}$, up to multiplication of the
quadratic form by a unit. Since $|d|$ is an $\ell$-adic unit, then 
\(
 \ord_\ell(\eta)=\ord_\ell(m).
\)
By \cite[Lemma~5.2]{KY}, we have 
\[
 e_\ell(m)
 =
 \sum_{a=0}^{\ord_\ell(m)}\chi_i(\ell)^a.
\]
For $\ell\nmid d_i$, equation \eqref{eq:epsilon-prime} gives
\(
 \chi_i(\ell)=\eps(\ell).
\)
At the coherent inert prime $p$, the switched quadratic space satisfies
\(
 \gamma\!\left(V_p^{(p)}\right)^{-1}
 W_{\eta,p}^{*}
 \left(0,\widetilde\varphi_p^{(p)}\right)
 =
 \mathbf{1}_{\Z_p}(\eta)=1
\)
by \cite[proof of Proposition~5.5(1)]{KY}.

Next suppose that $\ell\mid d$ and $\ell\nmid x$. Then
\(
 \ell\nmid m,
 \,
 \ord_\ell(\eta)=-1,
 \)
 and
 \(
  r_x\notin\OO_{i,\ell}.
\)
After the change of variables in \cite[Lemma~5.4]{KY}, the quadratic
coefficient is
\(
 t=-\frac14
\)
and the coset parameter is $-\overline{r}_x$. A direct calculation gives
\begin{equation}
 \eta-t r_x\overline{r}_x
 =
 \frac{1-d_j}{16}\in\Z_\ell.\nonumber
\end{equation}
Indeed, we have
\(
 r_x\overline{r}_x=\frac{d-x^2}{4d},
 \)
 and
 \(
 \eta=\frac{x^2-D}{16d}.
\)
In \cite[Lemma~5.4]{KY}, the conductor parameter is $0$, so the
normalized value is $1$. This agrees with \eqref{eq:finite-local-factor}
since $\ord_\ell(m)=0$.

Finally, suppose that $\ell\mid d$ and $\ell\mid x$. Since the
discriminants are coprime and fundamental, then
\(
 \ord_\ell(m)=1,
 \,
 \eta\in\Z_\ell^\times,
 \)
 and
 \(
 r_x\in\OO_{i,\ell}.
\)
By \cite[Lemma~5.3]{KY}, we have
\(
 e_\ell(m)=1+\chi_i(t\eta),
 \)
 and
 \(
 t=-\frac14.
\)
Writing $d=\ell d'$ and using
\(
 m=\frac{dd_j-x^2}{16},
\)
we obtain
\(
 t\eta=\frac{m}{4d}\equiv\frac{d_j}{64}\pmod\ell.
\)
The denominator is a square modulo $\ell$, and the ramified local
character on units is the Legendre symbol. Hence
\(
 \chi_i(t\eta)
 =
 \left(\frac{d_j}{\ell}\right)
 =
 \eps(\ell).
\)
Therefore
\(
 e_\ell(m)=1+\eps(\ell),
\)
which again agrees with \eqref{eq:finite-local-factor}.
\end{proof}

\begin{corollary}
\label{cor:e2}
For the fixed integer $m=m_x$, we have
\begin{equation}
 e_2(m)
 =
 \ord_2(m)+1
 =
 \rho\bigl(2^{\ord_2(m)}\bigr).\nonumber
\end{equation}
\end{corollary}

\begin{proof}
Since $d_i\equiv1\pmod8$,
\(
 K_i\otimes\Q_2\cong\Q_2\oplus\Q_2
\)
and
\(
 \chi_i(2)=1.
\)
Also, $2\nmid d_i$, so the local Schwartz function is
\(
 \mathbf{1}_{\OO_{i,2}}.
\)
The unramified formula in \cite[Lemma~5.2]{KY} gives
\[
 e_2(m)
 =
 \sum_{a=0}^{\ord_2(m)}\chi_i(2)^a
 =
 \ord_2(m)+1.
\]
Finally,
\(
 \eps(2)=\left(\frac{d_i}{2}\right)=1,
\)
so the last equality follows from \eqref{eq:rhoepsilon}.
\end{proof}

In what follows, we formulate the point-count of the special endomorphism cycle $\mathcal Z(m;\OO_i,\lambda,r_\xi)$ over~$\overline{\F}_{p}$.

\begin{proposition}
\label{prop:point-count}
For an admissible positive integer $m$, define
\begin{equation}
 \Diff(m)
 =
 \left\{
 \ell:
 \eps(\ell)=-1
 \text{ and }
 \ord_\ell(m)\text{ is odd}
 \right\}. \nonumber
\end{equation}

Fix any $x\in X_D$, write $m=m_x$, and let
\(
 \xi\in\{x,-x\}.
\)
The set of primes at which the local data fail to match in the
construction of \cite{KY} is exactly $\Diff(m)$. Consequently, the
special endomorphism cycle
\(
 \mathcal Z(m;\OO_i,\lambda,r_\xi)
\)
is empty in characteristic $p$ unless
\(
 \Diff(m)=\{p\}.
\)
If
\(
 \Diff(m)=\{p\}
\)
and
\(
 m=p^en,
 \)
 with
 \(
 p\nmid n,
\)
then $e$ is odd and
\begin{equation}
\label{eq:point-count}
 \#
 \mathcal Z(m;\OO_i,\lambda,r_\xi)(\overline{\F}_p)
 =
 \rho(n).
\end{equation}
Here $\#$ denotes the number of geometric points of the coarse finite
scheme, equivalently the number of isomorphism classes.
\end{proposition}

\begin{proof}
By \cite[Corollary~3.7]{KY}, a nonempty cycle is supported at the
unique prime in the corresponding difference set. The local
calculations in Proposition~\ref{prop:local-factors}, applied before
the coherent switch at $p$, show that the local Whittaker coefficient
vanishes exactly when
\(
 \eps(\ell)=-1
 \)
 and
 \(
 \ord_\ell(m)\text{ is odd}.
\)
Thus, the difference set is $\Diff(m)$.

Assume that
\(
 \Diff(m)=\{p\}.
\)
The coherent switch gives
\(
 e_p(m)=1.
\)
The archimedean calculation in
\cite[Proposition~2.6(i)]{KRY} gives
\(
 \gamma(V_\infty^{(p)})^{-1}
 W_{\eta,\infty}^{*}
 (\tau,0,\widetilde\varphi_\infty^{(p)})
 q^{-\eta}
 =
 2.
\)
The product formula for the local Weil indices, used in
\cite[proof of Proposition~3.5]{KY}, gives
\(
 \prod_{\ell\leq\infty}\gamma(V_\ell^{(p)})=1.
\)
Therefore,
\begin{equation}
 E_\eta^{*}(\tau,0;\widetilde\varphi^{(p)})q^{-\eta}
 =
 2\prod_{\ell<\infty}e_\ell(m)
 =
 2\prod_{\ell\ne p}e_\ell(m).\nonumber
\end{equation}

By \cite[Theorem~3.6]{KY}, the point count is $w_i/4$ times this
coefficient. Since
\(
 w_i=|\OO_i^\times|=2,
\)
the factor
\(
 \frac{w_i}{4}\cdot2
\)
is equal to $1$. Hence, we have
\[
\begin{aligned}
 \#
 \mathcal Z(m;\OO_i,\lambda,r_\xi)(\overline{\F}_p)
 &=
 \prod_{\ell\ne p}
 \sum_{a=0}^{\ord_\ell(m)}\eps(\ell)^a=
 \prod_\ell
 \rho\bigl(\ell^{\ord_\ell(n)}\bigr)=
 \rho(n).
\end{aligned}
\]\end{proof}

\begin{remark}
\label{rem:diff-cardinality}
Proposition~\ref{prop:genus-sign} implies that
$\Diff(m_x)$ has odd cardinality. Thus, the nonzero case of
Proposition~\ref{prop:point-count} is exactly the case in which the
difference set is a singleton. If the difference set has at least
three elements, the special endomorphism cycle is empty and the
corresponding divisor sums vanish.
\end{remark}

\section{Deformation lengths and lifting levels}
\label{sec:deformation-normalization}

Let $\mathcal Y_0(2)$ denote the Deligne--Mumford stack over
$\Z[1/2]$ classifying pairs $(E,C)$, where $E$ is an elliptic curve and
$C\subset E[2]$ is cyclic of order $2$. Let $Y_0(2)$ denote its coarse
moduli scheme.

\begin{definition}
\label{def:lifting-normalization}
Fix an odd prime $p$ and an $x\in X_D$. Suppose that there is an $i\in\{1,2\}$ such that
\(
 p\nmid d_i,
 \)
 and
 \(
 \left(\frac{d_i}{p}\right)=-1,
\)
and let $j\ne i$. Thus,
\(
 \OO_{i,p}/\Z_p
\)
is an unramified quadratic extension.

Write 
\(
 W=W(\overline{\F}_p)
\)
for the ring of Witt vectors of~$\overline{\F}_p$, 
and fix one of the two embeddings
\(
 \tau_i:\OO_{i,p}\hookrightarrow W.
\)
This fixes the CM type of the first CM action. Since
$\Z_p\to W$ is faithfully flat and unramified, then 
\[
 \length_W(M\otimes_{\Z_p}W)
 =
 \length_{\Z_p}(M)
\]
for any finite-length $\Z_p$-module $M$. Thus, fixing this CM type
introduces no numerical factor.

For $\xi\in\{x,-x\}$ and $n\ge1$, let
$\mathcal I_{p,n}^{\xi}(i)$ be the finite moduli problem of geometric
tuples
\(
 (E,C,\iota_i,\iota_j)
\)
in characteristic $p$ satisfying:
\begin{enumerate}[label=(\alph*)]
\item $E$ is an elliptic curve, $C\subset E[2]$ is cyclic of order $2$,
and $C$ is not stable under either maximal CM order;

\item $\iota_k:\OO_k\hookrightarrow\End(E)$ is an optimal embedding
for $k=i,j$, with the fixed CM type for $\iota_i$;

\item if
\(
 s_k=\iota_k(\sqrt{d_k}),
\)
then
\(
 \frac12\Trd(s_is_j)=\xi;
\)

\item the full tuple extends over the finite infinitesimal base
\(
 \Spec(W/(p^n)).
\)
\end{enumerate}
Morphisms are required to preserve all the displayed data.

For a finite moduli problem $\mathcal M$, define its degree by
\[
 \mdeg(\mathcal M)
 =
 \sum_{[z]\in\pi_0(\mathcal M)}
 \frac{1}{|\Aut(z)|}.
\]
Thus, $\mdeg(\mathcal M)$ is the usual degree with automorphisms
counted by their natural stack-theoretic weights. Set
\begin{equation}
\label{eq:L-normalization}
 L_{p,n}(x)
 =
 \mdeg\bigl(\mathcal I_{p,n}^{x}(i)\bigr)
 +
 \mdeg\bigl(\mathcal I_{p,n}^{-x}(i)\bigr).
\end{equation}
If no $i$ satisfies the displayed conditions, both finite moduli
problems are taken to be empty. The four-orientation comparison in the
proof of Lemma~\ref{lem:stack-normalization} shows that
\eqref{eq:L-normalization} is independent of the permissible choice of
$i$.
\end{definition}

Assume first that
\(
 \Diff(m_x)=\{p\}
\)
and write
\(
 m_x=p^en_x
 \)
 with
 \(
 p\nmid n_x.
\)
By assumption, $p$ is inert in the chosen base field $K_i$, and $e$
is odd.

\begin{proposition}
\label{prop:local-ring}
Let
\(
 \mathcal Z_W
 =
 \mathcal Z(m_x;\OO_i,\lambda,r_\xi)
 \times_{\Spec(\OO_{i,p}),\tau_i}
 \Spec(W).
\)
For any geometric point $z$ of $\mathcal Z_W$ in characteristic $p$, we have 
\begin{equation}
 \widehat{\OO}_{\mathcal Z_W,z}
 \cong W/(p^\nu),
 \qquad
 \nu=\frac{e+1}{2}.\nonumber
\end{equation}
\end{proposition}

\begin{proof}
After fixing the CM type, the unramified case of
\cite[Proposition~4.1]{KY} applies. In the notation of that result, the
relative Witt ring is $W$ and the uniformizer is $p$. Hence
\(
 \nu=\frac{\ord_p(m_x)+1}{2}.
\)
The deformation-theoretic input used there is
\cite[Proposition~4.3]{Gross}.
\end{proof}

For $\xi\in\{x,-x\}$ and $n\ge1$, let
\(
 \mathcal X_\xi^{[n]}(\overline{\F}_p)
\)
denote the finite moduli problem of geometric objects of
\(
 \mathcal Z(m_x;\OO_i,\lambda,r_\xi),
\)
whose completed local rings have $W$-length at least $n$.
Equivalently, these are the objects whose full signed CM and level
structure extends over
\(
 \Spec(W/(p^n)).
\)
Write
\(
 X_\xi^{[n]}(\overline{\F}_p)
 =
 \pi_0\left(
 \mathcal X_\xi^{[n]}(\overline{\F}_p)
 \right)
\)
for the set of isomorphism classes.

\begin{lemma}
\label{lem:stack-normalization}
For any $n\ge1$, we have 
\begin{equation}
\label{eq:normalization-explicit}
 L_{p,n}(x)
 =
 \#X_x^{[n]}(\overline{\F}_p).
\end{equation}
Therefore, if
\(
 \Diff(m_x)=\{p\},
\)
then
\begin{equation}
\label{eq:L-by-point-count}
 L_{p,n}(x)
 =
 \begin{cases}
 \#\mathcal Z(m_x;\OO_i,\lambda,r_x)(\overline{\F}_p),
   &1\le n\le(e+1)/2,\\[4pt]
 0,&n>(e+1)/2.
 \end{cases}
\end{equation}
\end{lemma}

\begin{proof}
By Definition~\ref{def:lifting-normalization},
Proposition~\ref{prop:reconstruct}, and Corollary~\ref{cor:level2}, we
have
\[
\begin{aligned}
 L_{p,n}(x)
 &=
 \sum_{[z]\in X_x^{[n]}(\overline{\F}_p)}
 \frac{1}{|\Aut(z)|}+
 \sum_{[z]\in X_{-x}^{[n]}(\overline{\F}_p)}
 \frac{1}{|\Aut(z)|}.
\end{aligned}
\]
Since $d_i\equiv1\pmod8$, we have
\(
 d_i\ne-3,-4,
 \)
 and
 \(
 \OO_i^\times=\{\pm1\}.
\)
An automorphism of a tuple must commute with the fixed $\OO_i$-action.
The centralizer of $K_i$ in the supersingular quaternion algebra is
$K_i$ itself, and optimality gives
\(
 K_i\cap\End(E)=\OO_i.
\)
Thus
\(
 \Aut(z)=\OO_i^\times=\{\pm1\}.
\)
Therefore
\(
 L_{p,n}(x)
 =
 \frac12\#X_x^{[n]}(\overline{\F}_p)
 +
 \frac12\#X_{-x}^{[n]}(\overline{\F}_p).
\)

Conjugating the second CM embedding sends
\[
 s_j\to-s_j,
 \qquad
 x\to-x,
 \qquad
 u\to-u.
\]
Equivalently, one has that
\(
 (E,\iota_i,u)\to(E,\iota_i,-u)
\)
gives an isomorphism
$$
 \mathcal Z(m_x;\OO_i,\lambda,r_x)
 \cong
 \mathcal Z(m_x;\OO_i,\lambda,r_{-x}).
$$
Indeed, equation \eqref{eq:reconstruct-two} changes
\(
 A_j=\frac{1+s_j}{2}
\)
into
\(
 1-A_j=\frac{1-s_j}{2}.
\)
This isomorphism preserves completed local rings and hence, all
extension conditions over the bases $\Spec(W/(p^n))$. Thus, the two
signed cycles have the same number of points at every lifting level.
It follows that
\(
 L_{p,n}(x)
 =
 \#X_x^{[n]}(\overline{\F}_p).
\)

We also verify that this paired degree does not depend on the permissible
structural field. Suppose that both $K_1$ and $K_2$ are unramified and
inert at $p$. A coarse simultaneous configuration has four oriented
versions
\(
 (\iota_1,\iota_2),\,
 (\overline{\iota}_1,\iota_2),\,
 (\iota_1,\overline{\iota}_2),\,
 (\overline{\iota}_1,\overline{\iota}_2).
\)
Conjugating exactly one action changes the signed trace from $\xi$ to
$-\xi$, whereas conjugating both actions leaves it unchanged. These
conjugation involutions identify the corresponding deformation problems.
When the structural action is conjugated, the identification is accompanied
by the corresponding unramified automorphism of $W$, so it preserves
$W$-length and every lifting condition over $W/(p^n)$. Fixing the
structural orientation of either field therefore leaves one object of sign
$x$ and one object of sign $-x$, each with weight $1/2$. Hence, the sum in
\eqref{eq:L-normalization} is independent of the permissible choice of
$i$.

Finally, Corollary~\ref{cor:level2} yields a unique common subgroup of
order $2$ that is not stable under either maximal CM order. Hence, no
 further level multiplicity occurs. Proposition~\ref{prop:local-ring}
now gives \eqref{eq:L-by-point-count}.
\end{proof}

\begin{proposition}
\label{prop:Lpn}
For any odd prime $p$ and every $x\in X_D$, we have 
\[
 \sum_{n\ge1}L_{p,n}(x)
 =
 \sum_{r\ge1}\rho(m_x/p^r).
\]
\end{proposition}

\begin{proof}
Suppose first that
\(
 \Diff(m_x)=\{p\}
\)
and write
\(
 m_x=p^{e_x}n_x
 \)
 with
 \(
 p\nmid n_x.
\)
By Proposition~\ref{prop:point-count},
Proposition~\ref{prop:local-ring}, and
Lemma~\ref{lem:stack-normalization}, every lifting level
\(
 1\le n\le\frac{e_x+1}{2}
\)
has contribution $\rho(n_x)$, and every later level has contribution
zero. Hence, one has that 
\[
 \sum_{n\ge1}L_{p,n}(x)
 =
 \frac{e_x+1}{2}\rho(n_x).
\]

By Proposition~\ref{prop:genus-sign},
\(
 \eps(m_x)=-1.
\)
Since $\eps(p)=-1$ and $e_x$ is odd, we have
\(
 \eps(n_x)=1.
\)
Lemma~\ref{lem:prime-power} therefore gives
\[
 \rho(m_x/p^r)
 =
 \begin{cases}
  \rho(n_x),&r\in\{1,3,\ldots,e_x\},\\
  0,&r\text{ is even or }r>e_x.
 \end{cases}
\]
There are exactly $(e_x+1)/2$ odd integers between $1$ and $e_x$.
In consequence,
\[
 \sum_{n\ge1}L_{p,n}(x)
 =
 \sum_{r=1}^{e_x}\rho(m_x/p^r)
 =
 \sum_{r\ge1}\rho(m_x/p^r).
\]

Now suppose that
\(
 \Diff(m_x)\ne\{p\}.
\)
By \cite[Corollary~3.7]{KY}, the signed cycles are empty in
characteristic $p$, so
\(
 L_{p,n}(x)=0
\)
for any $n$. Proposition~\ref{prop:genus-sign} implies that
$\Diff(m_x)$ has odd, and hence, positive, cardinality. Since it is not
$\{p\}$, it contains a prime $q\ne p$. For any $r\ge1$, either
$p^r\nmid m_x$, in which case $\rho(m_x/p^r)=0$ by convention, or the
exponent of $q$ remains odd in $m_x/p^r$. In the latter case,
Lemma~\ref{lem:prime-power} gives
\(
 \rho(m_x/p^r)=0.
\)
Thus, both sides of the asserted identity vanish.
\end{proof}


\section{Arithmetic intersections and lifting-order decompositions}
\label{sec:arithmetic-intersections}

In this section, we prove Theorem~\ref{thm:norm-formula} by establishing the following identity
\begin{equation}
\label{eq:global-comparisoninitial}
 \ord_p
 \left|
 \Norm_{\Q(\beta_1,\beta_2)/\Q}(\beta_1-\beta_2)
 \right|
 =
 \sum_{x\in X_D}\sum_{n\ge1}L_{p,n}(x) \nonumber
\end{equation}
for any odd prime $p$.

The proof needs three ingredients: a resultant expression for the field
norm, an integral intersection on $Y_0(2)$, and a decomposition of that
intersection according to the lifting order of supersingular
configurations. We start with the following three technical lemmas.

\begin{lemma}
\label{lem:local2}
Let
\(
 M\cong\Z_2^2,
\)
let $C\subset M/2M$ be a line, and let
\(
 A_i\in\End_{\Z_2}(M)
\)
satisfy
\(
 A_i^2-A_i+\frac{1-d_i}{4}I=0.
\)
Assume that $C$ is not stable under either $A_i$. Write 
\(
 s_i=2A_i-1,
 \,
 t=\tr(A_1A_2),
 \)
 and
 \(
 x=2t-1.
\)
Then
\begin{equation}
\label{eq:u-integral}
 u=\frac{s_1s_2-x}{4}\in\End_{\Z_2}(M),
\end{equation}
and $\tr(u)=0$ and 
 $\det(u)=\frac{D-x^2}{16}.$
In particular, we have 
\(
 x^2\equiv D\pmod{16}.
\)
\end{lemma}

\begin{proof}
Choose a basis of $M/2M$ in which
\(
 C=\langle e_1\rangle.
\)
Since $d_i\equiv1\pmod8$, the reduction
\(
 P_i=A_i\pmod2
\)
is an idempotent. The nonstability of $C$ means that the lower-left
entry of $P_i$ is $1$. Solving $P_i^2=P_i$ gives
\(
 P^{(a)}
 =
 \begin{pmatrix}
  a&0\\
  1&1+a
 \end{pmatrix},
 \)
 with
 \(
 a\in\F_2.
\)
If
\(
 P_1=P^{(a)},
 \)
 and
 \(
 P_2=P^{(b)},
\)
then
\(
 P_1+P_2=(a+b)I,
 \)
 and
 \(
 \tr(P_1P_2)=1+a+b.
\)

The nonstability of $C$ also shows that $A_i$ is not scalar. Hence, the
quadratic polynomial in the statement is both the minimal and the
characteristic polynomial of $A_i$. In particular, we have that 
\(
 \tr(A_i)=1,
 \)
 and
 \(
 \det(A_i)=\frac{1-d_i}{4}.
\)
Therefore, we have 
\(
 \frac12\tr(s_1s_2)
 =
 2\tr(A_1A_2)-1
 =
 x.
\)

Moreover,
\begin{equation}
\label{eq:divisibility-four}
 s_1s_2-xI
 =
 2\bigl(2A_1A_2-A_1-A_2+(1-t)I\bigr).
\end{equation}
The expression in parentheses reduces modulo $2$ to
\(
 P_1+P_2+
 \bigl(1-\tr(P_1P_2)\bigr)I
 =
 0.
\)
It is therefore divisible by $2$, so \eqref{eq:divisibility-four} is
divisible by $4$. This proves \eqref{eq:u-integral}.

Now
\(
 \tr(s_1s_2)=2x,
 \)
 and
 \(
 \det(s_1s_2)=(-d_1)(-d_2)=D.
\)
For a $2\times2$ matrix $T$, one has that 
\(
 \det(T-xI)
 =
 \det(T)-x\tr(T)+x^2.
\)
Taking $T=s_1s_2$ gives
\(
 \det(s_1s_2-xI)=D-x^2.
\)
Dividing by $4^2$ proves the determinant formula, and the trace
formula follows directly from the definition of~$x$.
\end{proof}

\begin{lemma}
\label{lem:resultant}
Let
\(
 M_i=\Q(\beta_i)
\)
and let $P_i(T)\in\Z[T]$ be the monic minimal polynomial of $\beta_i$.
Then:
\begin{enumerate}[label=(\roman*)]
\item
\(
 [M_i:\Q]=h(d_i),
\)
and $K_iM_i$ is the Hilbert class field of $K_i$;

\item
\(
 M_1\cap M_2=\Q;
\)

\item
\begin{equation}
 \left|
 \Norm_{M_1M_2/\Q}(\beta_1-\beta_2)
 \right|
 =
 \left|\Res(P_1,P_2)\right|.\nonumber
\end{equation}
\end{enumerate}
\end{lemma}

\begin{proof}
     See, e.g., \cite{Roskam, YY}.
\end{proof}


\begin{lemma}
\label{lem:modular-parameter-cm-divisors}
Over $\mathbb Z[1/2]$, the function
\[
t=\omega_2
\]
identifies the compactified coarse modular curve $X_0(2)$ with
$\mathbb P^1_{\mathbb Z[1/2]}$, and identifies the noncuspidal locus with
\[
Y_0(2)\simeq \mathbb G_{m,\mathbb Z[1/2]}
=\Spec\mathbb Z[1/2][t,t^{-1}].
\]
The two cusps are given by $t=0$ and $t=\infty$.

For $i\in\{1,2\}$, let $P_i(t)$ be the polynomial whose roots are the
Galois-conjugate CM values of $t$ associated with the discriminant $d_i$.
Then the horizontal CM divisor of discriminant $d_i$ on $Y_0(2)$ is cut
out by $P_i(t)$.

Moreover, for any odd prime $p$, the CM divisors have no support at the
cusps after base change to $\mathbb Z_p$. As a result, their local
intersection is represented by
\(
\Spec\left(
{\mathbb Z_p[t,t^{-1}]}/{(P_1(t),P_2(t))}
\right).
\)
\end{lemma}

\begin{proof}
The function $\omega_2$ is a principal modulus for $X_0(2)$ over
$\mathbb Q$. At the cusp $\infty$ one has
\[
\omega_2=2^{12}q+O(q^2),
\]
so $\omega_2$ is a uniformizer there up to a unit in
$\mathbb Z[1/2]$. The Fricke relation
\[
\omega_2\left(-\frac{1}{2z}\right)
=
\frac{2^{12}}{\omega_2(z)}
\]
shows that $\omega_2^{-1}$ is a uniformizer, again up to a unit, at the
other cusp. Thus $\omega_2$ is a modular unit on $Y_0(2)$ whose divisor on
$X_0(2)$ is the difference of the two cusp sections.

Since $\omega_2$ has degree one on every geometric fiber, it induces a
finite birational morphism
\[
X_0(2)\longrightarrow \mathbb P^1_{\mathbb Z[1/2]}.
\]
Both schemes are normal, so this morphism is an isomorphism. Removing the
two cusp sections gives
\[
Y_0(2)\simeq
\Spec\mathbb Z[1/2][t,t^{-1}].
\]

By construction, the roots of $P_i(t)$ are precisely the CM values of
the modular function $t$ associated with $d_i$. Hence the corresponding
horizontal CM divisor is cut out by $P_i(t)$.

It remains to check that, at an odd prime $p$, these divisors do not meet
the cusps. The CM values are algebraic integers and are $p$-adic units.
Equivalently, one may use the relation
\[
j\,t=(t+16)^3.
\]
The CM value of $j$ is integral. If $\ord_p(t)>0$, then the left-hand side
would have positive valuation, whereas $t+16$ is a $p$-adic unit because
$p$ is odd, so the right-hand side would have valuation zero. This is
impossible. Thus $\ord_p(t)=0$, and the CM points lie on the
$\mathbb G_m$-chart.
Therefore, the local intersection of the two CM divisors is represented by
\(
\Spec\left(
{\mathbb Z_p[t,t^{-1}]}/{(P_1(t),P_2(t))}
\right).
\)
\end{proof}

Upon the preceding lemma, one can interpret the $p$-adic order of $|\Res(P_1,P_2)|$ as an intersection length over~$\Z_p$.

\begin{proposition}
\label{lem:integral-intersection}
Let $p$ be an odd prime.
Then we have
\begin{equation}
    \label{eq:intersection-resultant}
\ord_p\bigl|\Res(P_1,P_2)\bigr|
=
\length_{\mathbb Z_p}
\left(
{\mathbb Z_p[t,t^{-1}]}/{(P_1(t),P_2(t))}
\right).    
\end{equation}

\end{proposition}

\begin{proof}
Note by the irreducibility of $P_1$ and $P_2$, and the proof of Lemma~\ref{lem:modular-parameter-cm-divisors}, one has that $\Res(P_1,P_2)\ne0$, and $P_{i}(0)\in\Z_p^{\times}$.
Denote
\(
A=\mathbb Z_p[t]/(P_1(t)).
\)
Since $P_1$ is monic, $A$ is finite free over $\mathbb Z_p$ of rank
$\deg P_1$. Multiplication by $P_2(t)$ defines a $\mathbb Z_p$-linear
endomorphism
\(
m_{P_2}:A\longrightarrow A.
\)
The fact of $\Res(P_1,P_2)\ne 0$ implies that $m_{P_2}$ is invertible
after tensoring with $\mathbb Q_p$. Its determinant is, up to sign,
the resultant:
\(
\det(m_{P_2})=\pm\Res(P_1,P_2).
\)
Moreover, one has 
\(
{\rm coker}(m_{P_2})
\simeq
{\mathbb Z_p[t]}/{(P_1(t),P_2(t))}.
\)

By the elementary divisor theorem over the discrete valuation ring
$\mathbb Z_p$, the valuation of the determinant of an endomorphism that
is invertible over $\mathbb Q_p$ equals the length of its cokernel.
Therefore, we have
\(
\ord_p\bigl|\Res(P_1,P_2)\bigr|
=
\length_{\mathbb Z_p}
\left({\mathbb Z_p[t]}/{(P_1,P_2)}\right).
\)

It remains to compare the polynomial and Laurent polynomial quotients.
Since $P_1(0)$ is a $p$-adic unit, the image of $t$ is invertible in $A$.
Indeed, writing
\[
P_1(t)=t^n+a_{n-1}t^{n-1}+\cdots+a_1t+a_0,
\]
with $a_0=P_1(0)\in\mathbb Z_p^\times$, the relation $P_1(t)=0$ in $A$
gives
\(
t\bigl(t^{n-1}+a_{n-1}t^{n-2}+\cdots+a_1\bigr)=-a_0.
\)
Since $-a_0$ is a unit, this shows that $t$ is a unit in $A$. Hence
localizing at $t$ does not change the quotient:
\[
{\mathbb Z_p[t]}/{(P_1,P_2)}
\cong
{\mathbb Z_p[t,t^{-1}]}/{(P_1,P_2)}.
\]
The desired equality follows.
\end{proof}


\begin{remark}
The use of $\mathbb Z_p[t,t^{-1}]$ reflects the modular interpretation:
the CM divisors lie on the noncuspidal curve
\[
Y_0(2)=\Spec\mathbb Z[1/2][t,t^{-1}],
\]
where $t=\omega_2$ is a modular unit. Thus the intersection is naturally
computed away from the cusp $t=0$. 
\end{remark}

To compute the intersection lengths, we first introduce the following arithmetic formula that can be viewed as the level-two analogue of Gross and Zagier's \cite[Proposition~2.3]{GZ}.

\begin{lemma}
\label{lem:formal-diagonal}
Let $A$ be a complete discrete valuation ring of mixed characteristic
$(0,p)$, where $p$ is odd, with algebraically closed residue field,
uniformizer $\pi$, and normalized valuation $v_A$. Let
\(
 y_k=(E_k,C_k)\in\mathcal Y_0(2)(A),
 \,
 t_k=\omega_2(y_k)
 \)
 for 
 \(
 k=1,2.
\)
Assume that the generic fibers of $y_1$ and $y_2$ are not isomorphic
and have automorphism group $\{\pm1\}$. Then
\begin{equation}
\label{eq:formal-diagonal-level2}
 v_A(t_1-t_2)
 =
 \frac12\sum_{r\ge1}
 \#\operatorname{Isom}_{A/(\pi^r)}
 \bigl((E_1,C_1),(E_2,C_2)\bigr).
\end{equation}

More generally, let $W$ be a complete unramified discrete valuation
ring with algebraically closed residue field, and let $D_1,D_2$ be
finite horizontal divisors on $Y_0(2)_W$ with disjoint generic fibers.
Write $I_W(D_1,D_2)$ for their total scheme-theoretic intersection
length over $W$.
Choose a finite extension $A/W$ of ramification degree
\(
 e_A=[A:W]
\)
over which all generic branches indexed by~$a$ and their CM actions are respectively represented by
sections $y_{k,a}$. Then one has that 
\begin{equation}
\label{eq:formal-diagonal-divisors}
 I_W(D_1,D_2)
 =
 \frac1{e_A}\sum_{a,b}v_A(t_{1,a}-t_{2,b})
 =
 \frac1{2e_A}
 \sum_{a,b}\sum_{r\ge1}
 \#\operatorname{Isom}_{A/(\pi^r)}
 (y_{1,a},y_{2,b}).
\end{equation}
Branches indexed by~$a,\,b$ are counted with their generic multiplicities. The expression
is independent of the chosen splitting extension $A$.
\end{lemma}

\begin{proof}
The corresponding formula for elliptic curves is
\cite[Proposition~2.3]{GZ}. See also \cite[Section~2]{LV} for a convenient restatement. It identifies the order of contact of two
sections of the coarse moduli line with one half of the total number
of infinitesimal isomorphisms.

Since $p\ne2$, the forgetful morphism
\(
 \mathcal Y_0(2)\longrightarrow\mathcal M_{\mathrm{ell}}
\)
is representable, finite, and $\acute{e}$tale . Thus, after an elliptic curve has
been lifted across a nilpotent thickening, each cyclic subgroup of
order~$2$ has a unique lift. Equivalently, for any $r$, the condition
\(
 g(C_1)=C_2
\)
defines an open-and-closed subscheme of the Gross--Zagier isomorphism
scheme over $A/(\pi^r)$. The diagonal of $\mathcal Y_0(2)$ is the
corresponding finite $\acute{e}$tale  pullback of the diagonal of
$\mathcal M_{\mathrm{ell}}$. Since the proof of
\cite[Proposition~2.3]{GZ} is local on this diagonal and additive over
its open-and-closed components, applying it to the level-preserving
component gives \eqref{eq:formal-diagonal-level2}.

The involution $g\to-g$ preserves $C_1$ and $C_2$ and acts freely,
which accounts for the factor $1/2$. Finally,
Lemma~\ref{lem:modular-parameter-cm-divisors} identifies the coarse curve with
$\mathbb P^1$ using the coordinate $t=\omega_2$. Thus, the local contact
order of the two level-two sections is
\(
 v_A(t_1-t_2).
\)

For the divisor statement, the resultant factorization after base
change to $A$ gives
\[
 e_A I_W(D_1,D_2)
 =
 \sum_{a,b}v_A(t_{1,a}-t_{2,b}).
\]
The factor $e_A$ is necessary when a horizontal CM branch is ramified
over $W$, since finite flat base change multiplies total intersection
length by $e_A$. Applying \eqref{eq:formal-diagonal-level2} to every
pair of branches proves \eqref{eq:formal-diagonal-divisors}. Passing
to a larger splitting field multiplies both the valuation sum and the
ramification degree by the same factor, so the expression is
independent of $A$.
\end{proof}

\begin{lemma}
\label{lem:lifting-support}
Let $p$ be an odd prime and assume that \(
 \Spec\left({\Z_p[t,t^{-1}]}/{(P_1(t),P_2(t))}\right)
\) is nonempty. Then every geometric point of
\(
 \Spec\left({\Z_p[t,t^{-1}]}/{(P_1(t),P_2(t))}\right)
\)
is supersingular. 

Moreover, on every nonempty component there is an
index $i\in\{1,2\}$ such that
\(
 p\nmid d_i,
 \)
 and
 \(
 \left(\frac{d_i}{p}\right)=-1.
\)
Fix such an index $i$, write $j\ne i$, and let
\(
 W=W(\overline{\F}_p).
\)
After the faithfully flat unramified base change
\(
 \Z_p\longrightarrow W,
\)
the CM divisor corresponding to $d_i$ splits into horizontal sections.
Fixing an embedding
\(
 \tau_i:\OO_{i,p}\hookrightarrow W
\)
selects one CM type above each such section and does not change
intersection lengths.
\end{lemma}

\begin{proof}
Suppose that a geometric point of the intersection corresponded to an
ordinary elliptic curve $E$. The two CM divisors would then give
embeddings of both $K_1$ and $K_2$ into $\End^0(E)$. The rational
endomorphism algebra of an ordinary elliptic curve is a single
imaginary quadratic field, so it cannot contain the two distinct
quadratic fields $K_1$ and $K_2$. Thus, every geometric intersection
point is supersingular.

At a supersingular point, the rational endomorphism algebra is the
quaternion algebra $B_{p,\infty}$, whose localization at $p$ is a
division algebra. In consequence, neither
\(
 K_k\otimes\Q_p
\)
can be split: a split quadratic algebra contains a nontrivial
idempotent, whereas a division algebra does not. Since
$(d_1,d_2)=1$, at most one of the two fields is ramified at $p$. Hence
at least one of them is unramified and inert at $p$, which gives the
asserted choice of $i$.

By Lemma~\ref{lem:resultant}(i), the field $K_iM_i$ is the Hilbert class
field of $K_i$. Since $K_i/\Q$ is unramified at $p$ and the Hilbert
class field is unramified over $K_i$, every localization of $M_i$ at
$p$ is unramified over $\Q_p$. Therefore the corresponding horizontal
CM divisor splits after the unramified extension $\Z_p\to W$.

The embedding $\tau_i$ specifies the action of $\OO_i$ on the Lie
algebra and hence, selects the corresponding CM type. Finally, for
every finite-length $\Z_p$-module $M$,
\(
 \length_W(M\otimes_{\Z_p}W)
 =
 \length_{\Z_p}(M),
\)
so the unramified base change preserves total intersection lengths.
\end{proof}

Combining the preceding two lemmas, one obtains a geometric formulation for the intersection lengths on the right hand side of~\eqref{eq:intersection-resultant}:

\begin{lemma}
\label{lem:intersection-stack}
Fix an odd prime $p$ and a permissible choice of $i$ as in
Lemma~\ref{lem:lifting-support}. Let $\mathcal T_p(i)$ be the finite
moduli problem of supersingular tuples
\(
 (E,C,\iota_i,\iota_j)
\)
with the fixed CM type for $\iota_i$, either CM type for the second
action, and with the two CM actions lying on the same intersection
component.

For any $z\in\mathcal T_p(i)$, let $R_z$ denote the complete
local deformation ring of the full tuple inside the intersection of
the two CM divisors. Then
\begin{equation}
\label{eq:intersection-stack-length}
 \length_{\Z_p}
 \left({\Z_p[t,t^{-1}]}/{(P_1(t),P_2(t))}\right)
 =
 \sum_{[z]\in\pi_0(\mathcal T_p(i))}
 \frac{\length_W(R_z)}{|\Aut(z)|}.
\end{equation}
\end{lemma}

\begin{proof}
Choose a finite totally ramified extension $A/W$, over which both
horizontal CM divisors, together with their CM actions, split into
sections. Write 
\(
 e_A=[A:W].
\)
By Lemma~\ref{lem:formal-diagonal},
\[
\length_{\Z_p}
 \left({\Z_p[t,t^{-1}]}/{(P_1(t),P_2(t))}\right)
 =
 \frac1{2e_A}
 \sum_{a,b}\sum_{r\ge1}
 \#\operatorname{Isom}_{A/(\pi^r)}
 (y_{1,a},y_{2,b}),
\]
where the sum ranges over the branches of the two divisors.

Grouping the level-preserving isomorphisms according to the tuple that
they induce gives the contribution of each object
$z\in\mathcal T_p(i)$. The factor $1/e_A$ descends the branchwise
$\pi$-adic lengths to $W$-lengths. Indeed, if the component attached
to $z$ becomes
\(
 R_z\widehat\otimes_W A,
\)
then
\(
 \length_A(R_z\widehat\otimes_W A)
 =
 e_A\length_W(R_z).
\)

It remains to account for the factor $1/2$ and the two CM types. The
involution
\(
 g\to-g
\)
acts freely on every level-preserving isomorphism set. The two
isomorphisms that transport a fixed second CM action to the same
embedding differ precisely by $\pm1$. Thus, the factor $1/2$ converts
isomorphisms into transported CM embeddings.

For a fixed CM type of the first action, a coarse second CM branch has
two possible CM actions, namely $\iota_j$ and
$\overline{\iota}_j$. Each oriented object has automorphism group
\(
 \Aut(z)=\OO_i^\times=\{\pm1\}.
\)
Indeed, an automorphism preserving the fixed $\OO_i$-action lies in
the centralizer of $K_i$, which is $K_i$, and optimality implies that
its integral elements are exactly $\OO_i$. Thus, the two CM types have
total degree
\(
 \frac12+\frac12=1,
\)
which agrees with the contribution of one $\{g,-g\}$-orbit. This proves
\eqref{eq:intersection-stack-length}.

The same argument shows that the formula is independent of the
splitting extension $A$.
\end{proof}

As we shall see in what follows, this will pass the computation to the countings of the geometric objects of~$\mathcal Z(m_x;\OO_i,\lambda,r_x)(\overline{\F}_p)$.

\begin{lemma}
\label{lem:signed-invariant}
Let
\(
 (E,C,\iota_1,\iota_2)
\)
be a simultaneous supersingular tuple occurring in $\mathcal T_p(i)$.
Write 
\(
 s_k=\iota_k(\sqrt{d_k}),
 \)
 and
 \(
 \xi=\frac12\Trd(s_1s_2).
\)
Then $\xi$ is an integer satisfying
\(
 0<|\xi|<\sqrt D,
 \)
 and
 \(
 \xi^2\equiv D\pmod{16}.
\)
In consequence, we have 
\(
 |\xi|\in X_D.
\)
Moreover, if
\(
 u=\frac{s_1s_2-\xi}{4},
\)
then $u\in\End(E)$ and
\(
 \Nrd(u)=\frac{D-\xi^2}{16}.
\)

Conversely, the tuples occurring in $\mathcal T_p(i)$ are precisely
those reconstructed from the special endomorphism cycles with signed
invariants $\xi=x$ and $\xi=-x$, for $x\in X_D$.
\end{lemma}

\begin{proof}
Set
\(
 A_k=\frac{1+s_k}{2}.
\)
Since $A_k$ is integral, then 
\[
 \xi
 =
 \frac12\Trd\bigl((2A_1-1)(2A_2-1)\bigr)
 =
 2\Trd(A_1A_2)-1
 \in\Z.
\]

The distinguished subgroup
\(
 C\subset E[2]
\)
is not stable under either maximal CM order by
Lemma~\ref{lem:odd-leading-line}. Since $p$ is odd, this property
persists on the $2$-adic Tate module $T_2E$. Applying
Lemma~\ref{lem:local2} to
\(
 M=T_2E
\)
gives
\(
 \frac{s_1s_2-\xi}{4}
 \in\End_{\Z_2}(T_2E),
\)
as well as
\(
 \Nrd(u)=\frac{D-\xi^2}{16},
 \)
 and
 \(
 \xi^2\equiv D\pmod{16}.
\)

The canonical map
\(
 \End(E)\otimes\Z_2
 \longrightarrow
 \End_{\Z_2}(T_2E)
\)
is an isomorphism. Indeed, it is injective, becomes an isomorphism
after tensoring with $\Q_2$, and both sides are maximal orders in
\[
 \End^0(E)\otimes\Q_2\simeq M_2(\Q_2).
\]
Thus, $u$ is integral at $2$. At every prime $\ell\ne2$, the denominator
$4$ is a unit, so $u$ is integral at $\ell$ as well. Intersecting the
local orders inside the rational endomorphism algebra gives
\(
 u\in\End(E).
\)

The element $u$ is nonzero, since otherwise the two quadratic
subfields generated by $s_1$ and $s_2$ would coincide. The supersingular
quaternion algebra is definite at infinity, so
\(
 \Nrd(u)>0.
\)
Therefore
\(
 |\xi|<\sqrt D.
\)
Since $D$ is odd, the congruence modulo $16$ implies that $\xi$ is odd,
and hence, nonzero. Thus, $|\xi|\in X_D$.

Conversely, Proposition~\ref{prop:reconstruct} reconstructs the second
maximal CM embedding from the special endomorphism $u$, and
Corollary~\ref{cor:level2} reconstructs the unique common cyclic
subgroup of order $2$ that is not stable under either maximal CM order.
Therefore the simultaneous tuples are exactly the tuples arising from
the two signed special endomorphism cycles.
\end{proof}

\begin{lemma}
\label{lem:lifting-length}
For any object $z$ in a signed component corresponding to
$x\in X_D$, let $R_z$ be its completed local deformation ring over
$W$. Then
\[
 \length_W(R_z)
 =
 \sum_{n\ge1}
 \mathbf{1}_{\{\length_W(R_z)\ge n\}}.
\]
Moreover, the condition
\(
 \length_W(R_z)\ge n
\)
is equivalent to the existence of an extension of the full signed
tuple over
\(
 \Spec(W/(p^n)).
\)
As an implication, we have 
\begin{equation}
\label{eq:completed-lengths-lifting-orders}
 \sum_{[z]\in\pi_0(\mathcal T_p(i))}
 \frac{\length_W(R_z)}{|\Aut(z)|}
 =
 \sum_{x\in X_D}\sum_{n\ge1}L_{p,n}(x).\nonumber
\end{equation}
\end{lemma}

\begin{proof}
If the signed special endomorphism cycle corresponding to $x$ is
nonempty in characteristic $p$, Proposition~\ref{prop:point-count}
implies that
\(
 \Diff(m_x)=\{p\}.
\)
Write
\(
 m_x=p^{e_x}n_x,
 \)
 and
 \(
 p\nmid n_x.
\)
Proposition~\ref{prop:local-ring} gives
\(
 R_z\simeq W/(p^{\nu_z}),
 \)
and
\(
 \nu_z=\frac{e_x+1}{2}.
\)
Hence, we have
\(
 \length_W(R_z)
 =
 \nu_z
 =
 \sum_{n\ge1}\mathbf{1}_{\{\nu_z\ge n\}}.
\)

The inequality
\(
 \nu_z\ge n
\)
is exactly the condition that the corresponding full tuple extends over
$\Spec(W/(p^n))$. Taking degrees and using Lemma~\ref{lem:signed-invariant} give
\[
\begin{aligned}
 \sum_{[z]\in\pi_0(\mathcal T_p(i))}
 \frac{\length_W(R_z)}{|\Aut(z)|}
 &=
 \sum_{x\in X_D}\sum_{n\ge1}
 \left(
 \mdeg\bigl(\mathcal I_{p,n}^{x}(i)\bigr)
 +
 \mdeg\bigl(\mathcal I_{p,n}^{-x}(i)\bigr)
 \right)=
 \sum_{x\in X_D}\sum_{n\ge1}L_{p,n}(x).
\end{aligned}
\]
If
\(
 \Diff(m_x)\ne\{p\},
\)
the corresponding signed cycles are empty in characteristic $p$ by
Proposition~\ref{prop:point-count}, so they contribute to neither side.
\end{proof}

\begin{lemma}
\label{lem:lifting-order-decomposition}
For any odd prime $p$, we have 
\begin{equation}
\label{eq:intersection-lifting-orders}
 \length_{\Z_p}
 \left({\Z_p[t,t^{-1}]}/{(P_1(t),P_2(t))}\right)
 =
 \sum_{x\in X_D}\sum_{n\ge1}L_{p,n}(x).
\end{equation}
\end{lemma}

\begin{proof}
If no permissible index exists, Lemma~\ref{lem:lifting-support} implies
that the intersection on the left of
\eqref{eq:intersection-lifting-orders} is empty. Its length is therefore
zero. By Definition~\ref{def:lifting-normalization}, the finite moduli
problems defining every $L_{p,n}(x)$ are then empty as well, so the
right-hand side is also zero.

Assume now that a permissible index exists, and choose $i$ as in
Lemma~\ref{lem:lifting-support}. By Lemma~\ref{lem:intersection-stack}, we have 
\[
 \length_{\Z_p}
 \left({\Z_p[t,t^{-1}]}/{(P_1(t),P_2(t))}\right)
 =
 \sum_{[z]\in\pi_0(\mathcal T_p(i))}
 \frac{\length_W(R_z)}{|\Aut(z)|}.
\]
Lemma~\ref{lem:lifting-length} identifies the right-hand side with
\(
 \sum_{x\in X_D}\sum_{n\ge1}L_{p,n}(x).
\)
This proves \eqref{eq:intersection-lifting-orders}.
\end{proof}

Finally, we are led up to the identity mentioned at the beginning of the current section.

\begin{proposition}
\label{prop:global-comparison}
For any odd prime $p$, we have 
\begin{equation}
\label{eq:global-comparison}
 \ord_p
 \left|
 \Norm_{\Q(\beta_1,\beta_2)/\Q}(\beta_1-\beta_2)
 \right|
 =
 \sum_{x\in X_D}\sum_{n\ge1}L_{p,n}(x).\nonumber
\end{equation}
\end{proposition}

\begin{proof}
This follows from combining Proposition~\ref{lem:integral-intersection} and Lemma~\ref{lem:lifting-order-decomposition}.
\end{proof}

\subsection{Proof of Theorem~\ref{thm:norm-formula}}
\label{sec:proof-norm-formula}

For an odd prime $p$, Propositions~\ref{prop:Lpn} and
\ref{prop:global-comparison} give
\[
 \ord_p
 \left|
 \Norm_{\Q(\beta_1,\beta_2)/\Q}(\beta_1-\beta_2)
 \right|
 =
 \sum_{x\in X_D}\sum_{r\ge1}\rho(m_x/p^r).
\]
Together with Proposition~\ref{prop:orderat2}, this shows that the two
positive rational numbers
\[
 \left|
 \Norm_{\Q(\beta_1,\beta_2)/\Q}(\beta_1-\beta_2)
 \right|
 \qquad\text{and}\qquad
 \prod_{x\in X_D}F(m_x)
\]
have the same valuation at every prime. 
This proves Theorem~\ref{thm:norm-formula}.\qed

\end{document}